\documentclass[review,authordate]{elsarticle}
\usepackage{amsmath,amsthm,eucal,mathbbol,eufrak,amscd,mathtools}
\usepackage{hyperref}
\providecommand\subject[1]{\medskip\noindent\textbf{Subject: }#1\\}
\providecommand\MSC[1]{\noindent\textbf{MSC: }#1\par\medskip}

\title{H-twisted Lie algebroids}
\author{Melchior Gr\"utzmann\fnref{fn1}}
\address{Department of Mathematics, Sun Yat-sen University, 135 Xingang Road West, Guangzhou 510275, P.R. China}
\ead{melchiorG@gMail.com}
\fntext[fn1]{partially supported by NSFC(10631050 and 10825105) and NKBRPC(2006CB805905)}

\newtheorem{defn}{Definition}
\newtheorem{lemma}[defn]{Lemma}
\newtheorem{prop}[defn]{Proposition}
\newtheorem{cor}[defn]{Corollary}
\newtheorem{thm}[defn]{Theorem}

{\theoremstyle{remark}
\newtheorem{rem}[defn]{Remark}
\newtheorem{ex}[defn]{Example}
}
\newcommand\defbb[2]{\def#1{{\mathbb{#2}}}}

\newcommand\defcal[2]{\def#1{{\mathcal{#2}}}}
\newcommand\deffrak[2]{\def#1{{\mathfrak{#2}}}}
\newcommand\defrm[2]{\def#1{{\mathrm{#2}}}}

\def\:{\colon}
\def\[{\begin{equation}}
\def\]{\end{equation}}
\def\<{\langle}  
\def\>{\rangle}
\def\conn_#1{\nabla_{\!\!#1\,}}
\defcal\A{A}
\def\alt{\text{alt.}}
\def\cycl{\text{cycl.}}
\defrm\ud{d}
\defrm\uD{D}
\DeclareMathOperator\Der{Der}
\let\oldepsilon=\epsilon
\let\epsilon=\varepsilon
\def\e{\oldepsilon}
\defcal\E{E}
\newcommand\Eforms[1][\bullet]{\Omega^{#1}_M(E)}
\newcommand\Esforms[2][\bullet]{\Omega^{#1}_M(E, #2)}
\def\Econn_#1{\conn_{#1}^E}
\def\Enabla{\nabla^E}
\def\Enconn_#1{\conn_{#1}^{E0}}
\def\ELie{\Lie^{\!E}}
\defcal\F{F}
\deffrak\g{g}
\def\Ht{{\tilde{H}}}
\DeclareMathOperator\Hom{Hom}
\def\im{i}
\defcal\Lie{L}
\defcal\M{M}
\defcal\N{N}
\defbb\NN{N}
\defcal\O{O}
\newcommand\pfrac[2]{\frac{\partial #1}{\partial #2}}
\DeclareMathOperator\pr{pr}
\def\Q{\Theta}
\def\tri{\triangleright}
\DeclareMathOperator\rk{rk}
\def\smooth{{\mathrm{C}^\infty}}
\def\gsmooth{{\mathcal{O}}}
\renewcommand\pmatrix[1]{\ensuremath{\left(\begin{array}{ccccc}#1\end{array}\right)}}
\deffrak\su{su}
\DeclareMathOperator\tr{tr}
\defcal\U{U}
\def\unsh{{\text{unshuffles}}}
\defrm\vol{vol}
\defcal\R{R}
\defbb\RR{R}
\deffrak\X{X}
\defbb\ZZ{Z}

\begin{document}
\maketitle

\begin{abstract}
We define a new kind of algebroid which fulfills a Leibniz rule, a Jacobi identity twisted by a 3-form $H$ with values in the kernel of the anchor map, and the twist is closed under a naturally occurring exterior covariant derivative.  We give examples and define three kinds of cohomology two via realization as Q-structure on graded manifolds.  The paper classifies PQ-manifolds of maximal degree 3.
\end{abstract}

\subject{Supermanifolds and Supergroups}
\MSC{17A32, 53D05, 58A50}

\section{Introduction}
A PQ-manifold is a graded manifold (ringed space with nonnegative integer graded structure sheaf plus some local trivialization conditions) with a symplectic structure and a compatible odd nilpotent vector field $Q$.  Due to a trick by Roytenberg the graded symplectic structures are all exact, as well as the $Q$-structures, which by definition have only to be symplectic vector fields, are all Hamiltonian.  If one starts with a study of the lowest degree PQ-manifolds one discovers the following structures.  A degree 1 P-manifold is an odd cotangent bundle, i.e.\ the fiber-linear functions on the cotangent bundle are declared to be of degree 1 while the coordinates of the base manifold remain of degree 0.  A Q-structure is now equivalent to a nilpotent Hamiltonian of degree 2.  Since the fiber-linear functions on $T^*M$ are the vector fields on $M$, a degree 2 function is a bivector field $\pi\in\Gamma(\wedge^2TM)$.  The Poisson bracket on $T^*M$ encodes the Schouten-Nijenhuis bracket of multivector fields and therefore nilpotence of $\pi$ is equivalent to $[\pi,\pi]=0$, i.e.\ the bivector is Poisson.  Therefore PQ-manifolds of degree 1 are in one-to-one correspondence to (smooth) Poisson manifolds.  Roytenberg described in \cite{Royt02} that a graded symplectic manifold of degree 2 is in one-to-one correspondence to a pseudo-Euclidean vector bundle.  He moreover discovered that Q-structures are in one-to-one correspondence to Courant algebroids on the underlying pseudo-Euclidean vector bundle.  The Dorfman bracket of the Courant algebroid can be reconstructed as a derived bracket between the graded functions of degree 1 (which are isomorphic to the sections of the vector bundle).  In the same paper Roytenberg also gave the elementary structure of graded manifolds as towers of affine fibrations.  

Based on that, the author investigated PQ-manifolds of degree 3 and we define an algebroid with structures similar to the PQ3-manifold.  It turns out that this generalization is an $H$-twisted Lie algebroid.  That is a vector bundle $E\to M$ (over a smooth manifold) together with a map of vector bundles $\rho\:E\to TM$, called the anchor, and a skew-symmetric bracket $[.,.]$ on the sections of $E$.  These fulfill the usual Leibniz rule.  More general than a Lie algebroid, the Jacobi identity can be twisted by an $E$-3-form $H$ with values in $\ker\rho$.  Guided by the PQ3-manifolds this $H$ should be closed under some exterior covariant derivative.  This arises naturally on the (possibly singular) subbundle $\ker\rho$ (see Section~2 for details).  Given this exterior covariant derivative we will define a cohomology by the na{\"\i}ve idea of cutting down the cochains ($E$-forms with values in $S^\bullet\ker\rho$) until it squares to 0.  This is summarized in the main Theorem~\ref{thm:ncoh} in Section~3.  The idea of twists of algebroids is well known in the literature, see e.g.\ \cite{YKS03}, \cite{YKS96b}, \cite[chap.6]{Xu97}, or \cite{CSX09b}.

The correspondence to PQ3-structures requires an additional structure -- a splitting.  This is summarized in Theorem~\ref{thm:PQ3} in Subsection~4.2.  Also for these split $H$-twisted Lie algebroids it is natural to define cohomology in terms of the Q-structure (analog to Courant algebroids).  We arrive thus at another definition of cohomology.

Note that the $H$-twisted Lie algebroid fits into the structure of a two-term $L_\infty$-algebra as introduced by \cite{Baez:03vi}.  We will review this definition in Section~3.  An interesting fact is that $H$-twisted Lie algebroids with an anchor map of constant rank can be described by $Q$-structures in another way, not involving a symplectic realization.  This observation goes back to the joint work of the author with T. Strobl in \cite{GSr10}.  This gives rise to another definition of cohomology in this regular case (details in Section~5).

Note that the realization as $Q$-structures naturally permits one to write down topological $\sigma$-models for $H$-twisted Lie algebroids.  They can therefore also be of interest for considerations in Quantum-field-theory.  Details of the underlying formalism can be found, e.g., in \cite{AKSZ97} or also in \cite{GSr10}.  Namely, during publication the author was informed of parallel developments by Ikeda and Uchino in \cite{Ikeda10b}, where particular PQ-manifolds of degree 3 as well as the sigma-models via AKSZ-construction are studied.  Their Lie algebroids up to homotopy correspond to splittable H-twisted Lie algebroids.

For those readers not interested in the graded geometry behind it, the Sections~2 and 3 are written entirely in terms of smooth geometry.  Also the na{\"\i}ve cohomology only relies on smooth vector bundles and the exterior covariant derivative on $\ker\rho$.  Since the na{\"\i}ve cohomology requires (smooth) sections of a potentially singular bundle (the kernel of the anchor map), it seems helpful to also give a definition of $H$-twisted Lie algebroids in terms of differential algebra.  This gives rise to what we want to call $H$-twisted Lie--Rinehart algebras.  For an introduction to Lie--Rinehart algebras see the original article \cite{Rin63} as well as \cite{CL07} for the notion of (co)-morphisms.

The organization of the paper is as follows.  In Section~2 we give a definition of the E-connection of an anchored almost Lie algebroid $(E,\rho)$ on $\ker\rho$.  We also explain what is meant by smooth sections in this possibly singular vector bundle.  In Section~3 we introduce the main object, $H$-twisted Lie algebroids.  We show that they are anchored (in the sense of Section~2) and show properties of its exterior covariant derivative.  We also give examples.  This permits us to define the na{\"\i}ve cohomology of an $H$-twisted Lie algebroid.  We end this section with the definition of $H$-twisted Lie--Rinehart algebras.  In Section~4 we take an excursion to graded symplectic Q-manifolds and show that PQ-manifolds of degree 3 (with splitting) give rise to split $H$-twisted Lie algebroids.  We moreover introduce the notion of split cohomology for splittable $H$-twisted Lie algebroids.  In Section~5, finally, we introduce a Q-structure for regular $H$-twisted Lie algebroids and define their regular cohomology.

I would like to thank Zhuo Chen for discussions about an early version of this paper and pointing out that this is a two-term $L_\infty$-algebra.  Moreover, I am thankful to D. Roytenberg for pointing me to the tangent complex (see Corollary~\ref{c:tc}) as well as the Example~\ref{ex:T^*EE}.  Additionally I would like to thank Z.-J. Liu and the referee for pointing me to Lie--Rinehart algebras.  Finally, I would also like to thank Yunhe Sheng of Jilin university for his hospitality during my stay there, since part of the work has been done during the stay.

\section{Connection for anchored almost Lie algebroids}
\begin{defn}  An anchored almost Lie algebroid $(E,\rho,[.,.])$ is a vector bundle $E\to M$, a map $\rho\:E\to TM$, called the anchor, and a skew-symmetric bracket $[.,.]\:\Gamma(E)\wedge\Gamma(E)\to\Gamma(E)$ subject to
\begin{align*}
  [\phi,f\cdot\psi] &= \rho(\phi)[f]\cdot\psi +f\cdot[\phi,\psi] \\
  \rho[\phi,\psi] &= [\rho(\phi),\rho(\psi)]_{TM}
\end{align*} for all $\phi,\psi\in\Gamma(E)$, and $f\in\smooth(M)$.
\end{defn}

\begin{ex}  Examples of these are Lie algebroids (which in addition fulfill the Jacobi identity), but also constructions of the following type.  Let $F\subset TM$ be an integrable distribution, $E:=F\oplus E_0$ be a vector bundle with projection $\rho\:E\to F\subset TM$, $\nabla$ a connection on $E_0$, $B\in\Esforms[2]{E_0}$, then
 $$ [X\oplus\phi,Y\oplus\psi]:=[X,Y]\oplus B(X\oplus\phi,Y\oplus\psi)+\conn_{\rho(X\oplus\phi)}\psi -\conn_{\rho(Y\oplus\psi)}\phi $$ is an anchored almost Lie algebroid.
\end{ex}
\begin{rem}  By a section $\Psi\in\Esforms{S^\bullet\ker\rho}$ of the possible singular ``bundle'' $\wedge^\bullet E^*\otimes S^\bullet\ker\rho$ -- $S^\bullet$ denoting the symmetric powers of a vector bundle -- we mean a smooth section of $\wedge^\bullet E^*\otimes S^\bullet E$ that vanishes under $\tilde\rho\:\wedge^\bullet E^*\otimes S^*E\to \wedge^\bullet E^*\otimes S^\bullet E\otimes TM$ defined using the Leibniz rule
$$ \tilde\rho(\alpha\otimes\psi_1\psi_2)=\alpha\otimes\big((\psi_1\otimes\rho(\psi_2))+(\psi_2\otimes\rho(\psi_1))\big) $$ where $\alpha\in\wedge^\bullet E_x^*$ and $\psi_i\in E_x$ for some $x\in M$.
\end{rem}

\begin{lemma}\label{l:kerrhoConn} There is an $E$-connection on $\ker\rho$, i.e.\ $\Enabla\:\Gamma(E)\otimes\Gamma(\ker\rho)\to\Gamma(\ker\rho)$ subject to
\begin{align*}
  \Econn_{f\phi}\psi &= f\Econn_\phi\psi \\
  \Econn_\phi(f\psi) &= \rho(\phi)[f]\cdot\psi+f\Econn_\phi\psi
\end{align*} for all $\phi\in\Gamma(E)$, $\psi\in\Gamma(\ker\rho)$, and $f\in\smooth(M)$.
\end{lemma}
\begin{proof} Define $\Econn_\phi\psi:=[\phi,\psi]$ for $\phi,\psi$ as above and observe that it lies in $\ker\rho$, because of the morphism property of $\rho$.  Since $\rho(\psi)=0$ it is also $\smooth(M)$-linear in $\phi$.
\end{proof}

\begin{defn}  Using the usual formula that works for Lie algebroids, we can extend $\nabla$ to an exterior covariant derivative $\uD\:\Esforms{\ker\rho}\to\Esforms[\bullet+1]{\ker\rho}$, i.e.\
\begin{align*}
  \<\uD\alpha,\psi_0\wedge\dots\psi_n\> :=& \sum_{i=0}^n (-1)^n\Econn_{\psi_i}\<\alpha,\psi_0\wedge\dots\hat{\psi}_i\dots\psi_n\> \\
   &+\sum_{i<j} (-1)^{i+j}\<\alpha,[\psi_i,\psi_j]\wedge\psi_0\dots\hat\psi_i\dots\hat\psi_j\dots\psi_n\>
\end{align*}
\end{defn}
\begin{lemma}  The above formula is indeed skew-symmetric and $\smooth(M)$-linear in all $\psi_i$.
\end{lemma}
\begin{proof} Straightforward computations analog to Lie algebroids.
\end{proof}

Note however that $\uD$, in general, does not square to $0$, because $[.,.]$ does not fulfill the Jacobi identity.  Namely the two are equivalent.

\begin{cor} $\Enabla$ and $\uD$ extend to $\Gamma(S^\bullet\ker\rho)$ and $\Esforms{S^\bullet\ker\rho}$, respectively.
\end{cor}
\begin{proof}  Extend using the Leibniz rule
\begin{align*}
  \Econn_\phi(\psi_1\psi_2) &:= (\Econn_\phi\psi_1)\psi_2 +\psi_1\Econn_\phi\psi_2
\end{align*}
$\uD$ extends in an analog way.
\end{proof}

\section{Definition, examples, elementary properties}
\begin{defn}\label{d:HLie}  An H-twisted Lie algebroid is a vector bundle $E\to M$ together with a bundle map $\rho\:E\to TM$ (called the anchor), a section $H\in\Esforms[3]{\ker\rho}$, and a skew-symmetric bracket $[.,.]\:\Gamma(E)\wedge\Gamma(E)\to\Gamma(E)$ subject to the axioms
\begin{align}
 [\phi,[\psi_1,\psi_2]] &= [[\phi,\psi_1],\psi_2] +[\psi_1,[\phi,\psi_2]] +H(\phi,\psi_1,\psi_2)  \label{Jacobi} \\
 [\phi, f\cdot\psi] &= \rho(\phi)[f]\cdot\psi +f\cdot[\phi,\psi] \label{Leibniz} \\
 \uD H &= 0 \label{Hclosed}
\end{align}  where $f\in\smooth(M)$, $\phi,\psi,\psi_i\in\Gamma(E)$ and $\uD$ is the one defined for anchored almost Lie algebroids.
\end{defn}

\begin{lemma}\label{l:rho}  $\rho$ is a morphism of brackets, i.e.\ 
\begin{align}
 \rho[\psi_1,\psi_2] &= [\rho(\psi_1),\rho(\psi_2)]
\end{align}  for $\psi_i\in\Gamma(E)$.
\end{lemma}
\begin{proof}  expand using $\rho(\psi)[f]\cdot\phi=[\psi,f\cdot\phi]-f\cdot[\psi,\phi]$, apply the Jacobi identity \eqref{Jacobi} and note that $\rho\circ H=0$ due to $H\in\Esforms[3]{\ker\rho}$.
\end{proof}

\begin{ex}\label{ex:rk3}  Let $(E,\rho,[.,.]_0)$ be a Lie algebroid of rank 3.  Define an E-connection $\nabla^{E0}$ on $\Gamma(S^\bullet \ker\rho)$ as in the Lemma~\ref{l:kerrhoConn} and an associated $\uD_0$.  

Take an E2-form $B\in\Esforms[2]{\ker\rho}$ and define
\begin{align}
  H &:= \uD_0 B, \\
  [\phi,\psi]_B &:= [\phi,\psi]_0+B(\phi,\psi),
\end{align}
{It follows that the Jacobi identity of $[.,.]_B$ is twisted by $\uD_0 B=H$.  Due to $\rk E=3$ we also observe that $\uD_B H=0$ for the differential defined in the Definition~\ref{d:HLie}.}

Therefore $(E,\rho,[.,.]_B,H)$ is an H-twisted Lie algebroid.
\end{ex}

\begin{ex}  We can generalize the last example if we take $\rk E\ge3$ with Lie algebroid structure $(E,\rho,[.,.]_0)$ and $\nabla^{E0}$, $\uD_0$ as before.  Starting from an arbitrary $B\in\Esforms[2]{\ker\rho}$ we observe again that $H:=\uD_0 B$ is the twist of the Jacobi identity.  For  $\alpha\in\Eforms[1]$ the operator $\uD$ is
\begin{align*}
  \uD \alpha = \uD_0\alpha-\alpha\circ B =: \uD_0\alpha+\tilde{B}(\alpha)
\intertext{and for $\phi\in\Gamma(\ker\rho)$}
  \uD \phi   = \uD_0\phi +B(\phi,\cdot)=: \uD_0\phi +\tilde{B}(\phi)
\end{align*}  and for arbitrary $\Psi\in\Esforms{S^\bullet \ker\rho}$ by extension by Leibniz rule.  The last condition \eqref{Hclosed} is therefore equivalent to
 $$ 0=\uD H = \uD_0(\uD_0 B)+\tilde{B}(\uD_0 B)= \tilde{B}(\uD_0 B) \;,
 $$ where we used that the differential $\uD_0$ of the (flat) module $\ker\rho$ squares to 0.  Therefore we can generalize the previous example if we can solve this quadratic first order (partial) differential equation.
\end{ex}

From \cite{Baez:03vi} we take the following definition of a two-term $L_\infty$-algebra.
\begin{defn}  A two-term $L_\infty$-algebra is a two-term complex $0\to V_1\xrightarrow{\partial} V_0\to 0$ together with three more maps
\begin{align}
  [.,.]\:V_0\wedge V_0&\to V_0, \nonumber\\
  \tri\:V_0\otimes V_1&\to V_1, \nonumber\\
  l_3\: V_0\wedge V_0\wedge V_0 &\to V_1  \nonumber
\intertext{Subject to the rules}
  [\phi,\partial f] &= \partial(\phi\tri f)  \label{n=2}\\
  (\partial f)\tri g + (\partial g)\tri f &= 0  \label{n=2b}\\
  [\phi_1,[\phi_2,\phi_3]\,] +\cycl &= \partial l_3(\phi_1,\phi_2,\phi_3)  \label{n=3}\\
  \phi_1\tri(\phi_2\tri f) -\phi_2\tri(\phi_1\tri f) -[\phi_1,\phi_2]\tri f &= l_3(\phi_1,\phi_2,\partial f)  \label{n=3b}\\
  l_3([\phi_1,\phi_2]\wedge\phi_3\wedge\phi_4) +\phi_1\tri l_3(\phi_2&\wedge\phi_3\wedge\phi_4) +\unsh = 0   \label{n=4}
\end{align} where $\phi_i\in V_0$ and $f\in V_1$.
\end{defn}

\begin{prop}  The $H$-twisted Lie algebroid $(E,\rho,[.,.],H)$ is a two-term $L_\infty$-algebra with the following identifications:
\begin{gather*}
  0\to \Gamma(\ker\rho)=:V_1\xhookrightarrow{\partial} \Gamma(E)=:V_0 \to 0\quad\text{a complex}, \\
  l_2\:V_0\wedge V_0\to V_0 : (\phi,\psi)\mapsto [\phi,\psi] \\
  \tri\:V_0\otimes V_1\to V_1 : (\phi,\psi)\mapsto \Econn_\phi\psi \\
  l_3\:V_0\wedge V_0\wedge V_0 \to V_1 : (\psi_0,\psi_1,\psi_2)\mapsto H(\psi_0,\psi_1,\psi_2)
\end{gather*}
\end{prop}
\begin{proof}  \eqref{n=2} follows from the definition of the connection, \eqref{n=3} is the axiom of homotopy Jacobi identity \eqref{Jacobi}, \eqref{n=4} can be identified with the closeness of $H$ under the derivative $\uD$ \eqref{Hclosed}.  \eqref{n=2b} follows from the skew-symmetry of the bracket (and the definition of the connection $\Enabla$) and \eqref{n=3b} for $\psi_i\in\Gamma(E)$ and $\phi\in\Gamma(\ker\rho)$ derives as follows:
\begin{align*}
&(\Econn_{\psi_1}\Econn_{\psi_2}-\Econn_{\psi_2}\Econn_{\psi_1}-\Econn_{[\psi_1,\psi_2]})\phi  \\
  =& [\psi_1,[\psi_2,\phi]] - [\psi_2,[\psi_1,\phi]] - [[\psi_1,\psi_2],\phi] \\
  =& H(\phi,\psi_1,\psi_2) \;.
\end{align*}
\end{proof}

\subsection{Na{\"\i}ve cohomology}
\begin{prop}\label{p:curv}  The operator $\uD$ fulfills a Leibniz rule, i.e.\ for $\alpha\in\Esforms[|\alpha|]{S^\bullet\ker\rho}$, $\beta\in\Esforms{S^\bullet \ker\rho}$ it is
 \[ \uD(\alpha\wedge\beta) = (\uD\alpha)\wedge\beta +(-1)^{|\alpha|}\alpha\wedge\uD\beta \;.
 \]
 \[    \uD^2(\alpha_1\wedge\dots\wedge\alpha_k) = (\uD^2\alpha_1)\wedge\alpha_2\wedge\dots\wedge\alpha_k+\dots+\alpha_1\wedge\dots\alpha_{k-1}\wedge(\uD^2\alpha_k)
 \]
  Note that the operator $\uD$ does not square to 0 in general.  It is
\begin{align}
  \uD^2f &= 0, \\
  \uD^2\alpha_1 &= -\alpha_1\circ H=:\Ht(\alpha_1), \nonumber\\
  \<\uD^2\phi,\psi_1\wedge\psi_2\> &= H(\phi,\psi_1,\psi_2)=:\<\tilde{H}(\phi), \psi_1\wedge\psi_2\>  \label{D^2phi}
\end{align} for $f\in\smooth(M)$, $\alpha_i\in\Eforms[1]=\Gamma(E^*)$, $\phi\in\Gamma(\ker\rho)$, and $\psi_i\in\Gamma(E)$.
\end{prop}
\begin{proof} For the first statement, note that $\uD$ is an odd first order linear differential operator.

For the second statement, observe
\begin{align*}
  \uD^2(\alpha_1\wedge\alpha_2) &= (\uD^2\alpha_1)\wedge\alpha_2 +(-1)^{|\alpha_1|}(\uD\alpha_1\wedge\uD\alpha_2-\uD\alpha_1\wedge\uD\alpha_2) \\&\quad+\alpha_1\wedge\uD^2\alpha_2
\intertext{and conclude for $k$ terms $\alpha_i$ by induction.  It thus remains to show the formulas for $f$, $\alpha$, and $\phi\in\Gamma(\ker\rho)$.}
 \<\uD^2 f,\psi_1,\psi_2\> &= \rho(\psi_1)[\rho(\psi_2)[f]] -\rho(\psi_2)[\rho(\psi_1)[f]] -\rho[\psi_1,\psi_2])[f]  \\
  &= ([\rho(\psi_1),\rho(\psi_2)]-\rho[\psi_1,\psi_2])[f] = 0
\intertext{due to Lemma~\ref{l:rho}.  For $\alpha\in\Eforms[1]=\Gamma(E^*)$ we have}
  \<\uD^2\alpha,\psi_0,\psi_1,\psi_2\> =& \rho(\psi_0)[\<\uD\alpha,\psi_1,\psi_2\>] -\<\uD\alpha,[\psi_0,\psi_1],\psi_2\> +\cycl \\
  =& (\rho(\psi_0)\rho(\psi_1)-\rho(\psi_1)\rho(\psi_0)-\rho[\psi_0,\psi_1])[\<\alpha,\psi_2\>] +\cycl+\\
  &+\<\alpha,[[\psi_0,\psi_1],\psi_2]+\cycl\>
\end{align*}
{The first term vanishes due to Lemma~\ref{l:rho}.  The second term simplifies using the Jacobi identity and gives the claim.}  \eqref{D^2phi} is property \eqref{n=3b} of a two-term $L_\infty$-algebra.
\end{proof}

\begin{rem}  
{Extend moreover the operator $\Ht$ from the last proposition}
\begin{align}
  \Ht&\:\Esforms[p+1]{S^q\ker\rho}\to \Esforms[p+3]{S^q\ker\rho} : \tilde{H}(\alpha_0\wedge\dots\alpha_{p}\otimes\psi_1\cdots\psi_q) =\\
  =& \sum_{i=0}^p (-1)^i\tilde{H}(\alpha_i)\wedge\alpha_0\dots\hat{\alpha}_i\dots\alpha_p\otimes\psi_1\cdots\psi_q \nonumber\\
  &+\sum_{j=1}^q \alpha_0\wedge\dots\alpha_p\wedge\tilde{H}(\psi_j)\cdot\psi_1\dots\hat{\psi}_j\cdots\psi_q \;.  \nonumber
\intertext{Remember the trace operator.  Given a tensor product of vectors and covectors, e.g.\ $\wedge^pE^*\otimes S^qE$ it contracts one vector with one covector, thus lowering the number of factors by one each.  In explicit formulas this is}
 \tr&\:\wedge^{k+1}E^*\otimes S^m E \to \wedge^kE^*\otimes S^{m-1}E :
  \tr(\alpha_0\wedge\dots\alpha_k\otimes\psi_1\cdots\psi_m) = \\
  &= \sum_{i=0}^k\sum_{j=1}^m (-1)^i \<\alpha_i,\psi_j\>\cdot\alpha_0\wedge\dots\hat{\alpha}_i\dots\alpha_k\otimes\psi_1\cdots\hat{\psi}_j\cdots\psi_m \;. \nonumber
\end{align}
\end{rem}

\begin{thm}[naive cohomology]\label{thm:ncoh}  Given an H-twisted Lie algebroid $(E,\rho,[.,.],H)$ we define its naive cochains as 
\begin{align}  C^{p,q}(E) := \ker&\bigl(\tilde{H}|\Esforms[p]{S^{q} \ker\rho}\bigr)
\intertext{and the differential}
  \ud_E\:C^{p,\bullet}(E) &\to C^{p+1,\bullet}(E): \Psi\mapsto \uD\Psi, \\
\intertext{together with an accompanying differential}
  \delta\:C^{p,q}(E) &\to C^{p-1,q-1}(E): \Psi\mapsto \tr\Psi
\end{align}
\end{thm}
Its partial cohomology $H^\bullet_{naive}(E):=H^\bullet(C^{\bullet,\bullet}(E),\ud_E)$ is called the na{\"\i}ve cohomology.
\begin{proof}
\begin{lemma} $\uD\circ\Ht -\Ht\circ\uD=\widetilde{\uD H}$ and $\Ht\circ\tr=\tr\circ\Ht$ \qed
\end{lemma}
Now for $\Psi\in\ker\Ht$, $\uD^2\Psi=\Ht(\Psi)=0$ and 
$$\tr^2(\alpha_1\wedge\alpha_2\otimes\psi\psi)=\<\alpha_1,\psi\>\<\alpha_2,\psi\> -\<\alpha_2,\psi\>\<\alpha_1,\psi\>=0 $$ and by a polarization argument also for two arbitrary $\psi_1$ and $\psi_2$.  By an induction argument this extends to arbitrary $\Psi\in\Esforms[p]{S^q\ker\rho}$.
\end{proof}
Note that the two differentials do not interchange, i.e.\ this is not a double complex. The name na{\"\i}ve cohomology is in analogy to \cite{SX07}, because we cut down the cochains such that $\uD$ squares to 0.  It might thus be that $H^{2,0}_{naive}(E)\oplus H^{0,1}_{naive}(E)$ does not cover all infinitesimal deformations of $E$.

\begin{ex}\begin{enumerate}\item  If $(E,\rho,[.,.])$ is a Lie algebroid (and $H$ thus vanishes), the na{\"\i}ve cohomology coincides with Lie algebroid cohomology with coefficients in $S^\bullet\ker\rho$.
\item  Let $\g_0=\su(2)$ and $B:=\xi^1\xi^2\otimes X_1$.  Then $H=\vol\otimes X_2$ and 
 $$ C^{p,q}(\g,H) = \begin{cases} \wedge^p\g^*\otimes S^q\g \quad\text{for }p\ne1, \\
   \RR\xi^1\oplus\RR\xi^3  \quad\text{for }p=1,q=0, \\
   \left\{\pmatrix{a&0&b\\ c&d&e\\ f&0&-a}:a,\dots,f\in\RR\right\}  \quad\text{for }p=1=q,\\
    \dots\quad\text{for }p=1,n_1+n_2+n_3=q-1\ge1
 \end{cases}
 $$  The cohomology of $\su(2)$ is well known due to Whitehead's lemma (see e.g.\ \cite{Vara}), i.e.\ $H^p(\su(2))=0$ for $p=1,2$ and $\RR$ for $p=0,3$.  For the $B$-twisted algebra however this is
 $$ H^{p,0}_{naive}(\g,B) \cong \begin{cases} \RR  \quad\text{for }p=0,2,3 \\
   0  \quad\text{for }p=1\;,
 \end{cases}
 $$ i.e.\ the twist has an effect on $H^{2,0}(\g,B)$.  The cohomology with coefficients in $\g$ is
 $$ H^{p,1}_{naive}(\g,B) \cong \begin{cases} 0\quad\text{for }p=0,1,3 \\
   \RR^2\quad\text{for }p=2 \;.
 \end{cases}
 $$ as opposed to Whitehead's theorem (see, e.g.\ \cite[WhiteheadsLemmas]{planetmath}) for $\su(2)$, where the Lie algebra cohomology with coefficients in an irreducible module of rank at least 2 vanishes.
\end{enumerate}
\end{ex}

\subsection{$H$-twisted Lie--Rinehart algebras}
Since the cochains $C^{p,q}:=\Esforms[p]{S^q\ker\rho}$ are smooth sections of a possibly singular vector bundle, the more natural language for the $H$-twisted Lie algebroids is that of modules over (smooth) algebras.  Throughout this chapter $\R$ will be a commutative (associative) ring with unit 1 over a base field $k$.  We will denote by $\Der(\R)$ the derivations of $\R$ (being additive and fulfilling the usual Leibniz rule).  By $\otimes$ we denote the usual tensor product over $k$.  Given an $\R$-module $\E$, we denote $\E^*:=\Hom_\R(\E,\R)$ its dual module.  Remember the definition of Lie--Rinehart algebra \cite{Rin63,CL07}.
\begin{defn}  A Lie--Rinehart algebra $(\R,\E,[.,.],\rho)$ is an $\R$-module $\E$ that has the structure of a Lie algebra $(\E,[.,.])$ together with an $\R$-linear map $\rho\:\E\to\Der(\R)$ subject to the rules
\begin{align*}
  [\phi,[\psi_1,\psi_2]] &= [[\phi,\psi_1],\psi_2]+[\psi_1,[\phi,\psi_2]] \\
  [\phi, f\cdot\psi] &= \rho(\phi)[f]\cdot\psi +f\cdot[\phi,\psi] \\
  \rho[\phi,\psi] &= [\rho(\phi),\rho(\psi)]_{\Der(\R)}
\end{align*} where $\phi,\psi_i\in\E$, $f\in\R$.
\end{defn}
Note that the first axiom is the Jacobi-identity of the Lie algebra.  The second axiom is the Leibniz rule and the third axiom is the morphism-property for the anchor map $\rho$.  The last axiom follows from the first two for projective modules $\E$, however one usually does not make this restriction in the definition.

The analogy for $H$-twisted Lie algebroids is now.
\begin{defn}  An $H$-twisted Lie--Rinehart algebra $(\R,\E,[.,.],\rho,H)$ is an $\R$-module $\E$, a skew-symmetric $k$-linear bracket $[.,.]\:\E\wedge\E\to\E$, an $\R$-linear map $\rho\:\E\to\Der(\R)$, and an E-3-form $H$ with values in $\ker\rho$, i.e.\ $H\in\Hom_\R(\wedge^3\E,\ker\rho)$, subject to the rules
\begin{align}
  [\psi_1,[\psi_2,\psi_3]]+\cycl &= H(\psi_1,\psi_2,\psi_3) \\
  [\phi,f\cdot\psi] &= \rho(\phi)[f]\cdot\psi +f\cdot[\phi,\psi] \\
  \rho[\phi,\psi] &= [\rho(\phi),\rho(\psi)] \\
  \uD H &= 0\;,
\end{align} where $\phi,\psi_i\in\E$, $f\in\R$ and $\uD\:\Hom_\R(\wedge^\bullet\E,\ker\rho)\to\Hom_\R(\wedge^{\bullet+1}\E,\ker\rho)$ is the exterior covariant derivative induced by the $\E$-connection $\nabla\:\E\otimes\ker\rho\to\ker\rho$ which is induced by the bracket $[.,.]$.
\end{defn}
The construction of the $\E$-connection on $\ker\rho$ as well as the exterior covariant derivative is analog to the algebroid case.

\begin{ex} A big class of examples comes from $H$-twisted Lie algebroids, where $\R=\smooth(M)$, $\E=\Gamma(E)$, and $[.,.]$ and $H$ are the corresponding structures.  In particular $H$-twisted Lie algebras are 2-term $L_\infty$-algebras with the differential $\partial\:V_1\to V_0$ being an embedding.
\end{ex}

In order to form a category we also need to specify the morphisms.  These come in two versions.
\begin{defn}  Given two $H$-twisted Lie--Rinehart algebras $(\R_i,\E_i,[.,.]_i,\rho_i,H_i)$, $i=1,2$ together with their induced connections $\nabla$ and $\nabla'$, an $L_\infty$-morphism $(\Phi^*,\Phi_1^*,\Phi_2^*)$ of Lie--Rinehart algebras is a morphism of base rings $\Phi^*\:\R_1\to\R_2$, a morphism of modules $\Phi_1^*\:\E_1\to\E_2$ and $\Phi_2^*\:\wedge^2\E_1\to\ker\rho_2$, subject to the rules
\begin{align}
  \Phi^*(f)\cdot\Phi_1^*(\psi) &= \Phi_1^*(f\cdot\psi),\\
  \Phi^*(f_1f_2)\cdot\Phi_2^*(\psi_1,\psi_2) &= \Phi_2^*(f_1\cdot\psi_1,f_2\cdot\psi_2),\\
  \Phi_1^*[\psi_1,\psi_2]_1-[\Phi_1^*\psi_1,\Phi_1^*\psi_2]_2 &= \Phi_2^*(\psi_1,\psi_2), \\
  \Phi^*(\rho_1(\psi)[f])-\rho_2\circ\Phi_1^*(\psi)[\Phi^*f] &= 0,\\
  \Phi_1^*\circ H_1(\psi_1,\psi_2,\psi_3) -H_2\circ\wedge^3\Phi_1^*(\psi_1,\psi_2,\psi_3) &= \conn_{\Phi_1^*\psi_1}\Phi_2^*(\psi_2,\psi_3)
  \\&\quad-\Phi_2^*([\psi_1,\psi_2]_1,\psi_3)+\cycl\nonumber
\end{align}  For all $\psi_i\in\E_1$, $f_i\in\R_1$.

We call such an $L_\infty$-morphism strict iff $\Phi_2^*\equiv0$.
\end{defn}
Note that despite the notation we do not require the $\Phi_i^*$ to be transposes of any maps $\Phi_i$.
\begin{prop} The $H$-twisted Lie--Rinehart algebras together with $L_\infty$-morphisms form a category.  The composition law for two morphisms $\Phi^*\:(\R_2,\E_2)\to(\R_3,\E_3)$ and $\Psi^*\:(\R_1,\E_2)\to(\R_2,\E_2)$ is
\begin{align}  (\Phi^*\circ\Psi^*) :\R_1&\to \R_3, \\
  (\Phi^*\circ\Psi^*)_1 :\E_1&\to\E_3: \chi\mapsto \Phi_1^*\circ\Psi_1^*(\chi), \\
  (\Phi^*\circ\Psi^*)_2 :\wedge^2\E_1&\to\ker\rho_3 : (\chi_1,\chi_2)\mapsto \Phi_2^*\circ\wedge^2\Psi_1^*(\chi_1,\chi_2)+\Phi_1^*\circ\Psi_2^*(\chi_1,\chi_2)
\end{align} where $\chi_i\in\E_1$, $f\in\R_1$, and $\rho_3$ is the anchor map of $(\R_3,\E_3)$.

The strict morphisms for a subcategory.
\end{prop}

The result is that an $L_\infty$-morphism of $H$-twisted Lie algebroids is the following.
\begin{defn}  Let $(E_i\to M_i,[.,.]_i,\rho_i,H_i)$, $i=1,2$ be two $H$-twisted Lie algebroids.  A morphism between them is a triple of maps $\Phi\:M_2\to M_1$, $\Phi_1\:E_2^*\to E_1^*$, and $\Phi_2:(\ker\rho_2)^*\to\wedge^2E_1^*$, such that their transpose $(\Phi^*,\Phi_1^*,\Phi_2^*)$ form an $L_\infty$-morphism of Lie--Rinehart algebras.  We call it strict morphism iff $\Phi_2\equiv0$.
\end{defn}
\begin{cor}  Also the $H$-twisted Lie algebroids together with $L_\infty$-morphisms form a category.  Moreover the strict morphisms form a subcategory.
\end{cor}

\section{Splittable H-twisted Lie algebroids}
\subsection{Motivation from PQ-manifolds}
The author's motivation to investigate an H-twist came from a generalization of \cite{Royt02} and \cite{GSr10} to degree 3 PQ graded manifolds.  For a short introduction to graded manifolds, see e.g.\ \cite[Chap.~2.2]{Gru09}, \cite[Sect.~2]{Royt02}, \cite[Sect.~2]{Sev01b}, or \cite[Sect.~4]{Vor01}.

Remember the Euler vector field $\e$ of a graded manifold $\M$.  Its eigenfunctions are the homogeneous functions on $\M$.

\begin{prop}\label{p:PQ3}  A symplectic N-manifold of degree 3 has an exact symplectic form $\omega=\ud(\tfrac13\im_\e\omega)=\ud \theta_i\wedge \ud x^i +\ud b_a\wedge\ud\xi^a$\footnote{$x^i$ of degree 0, $\xi^a$ of degree 1, $b_a$ symplectic duals of $\xi^a$ and of degree 2, and $\theta_i$ the symplectic duals of the $x^i$ and of degree 3} and locally the structure $\M\approx T^*[3]E[1]$.

It fits in the short exact sequences \[ 0\times M\to T^*[3]M\xrightarrow{i_p} \M\xrightarrow{p} \E\to 0\times M \;\] and $$0\times M\to E^*[2]\xrightarrow{i_q} \E\xrightarrow{q} E[1]\to 0\times M$$ of pointed graded fiber bundles over $M$ the body of $\M$.  \qed
\end{prop}
This is an observation by Roytenberg in \cite{Royt02}.  Canonical coordinate changes read as:
\begin{align*}
  \tilde{x}^j &= \tilde{x}^j(x), \\
  \tilde{\xi}^a &= M^{\tilde a}_b(x)\xi^b, \\
  \tilde{b}_a &= (M^{-1})^b_{\tilde{a}}b_b +\tfrac12 R_{\tilde{a}bc}(x)\xi^a\xi^b, \\
  \tilde{\theta}_i &= \pfrac{x^j}{\tilde{x}^i}\left(\theta_j- M_{a,j}^{\tilde{c}}(M^{-1})^b_{\tilde c}\xi^a b_b +\tfrac16R_{\tilde{d}bc,j}M^{\tilde{d}}_a\xi^a\xi^b\xi^c  \right),
\intertext{where $R\in\Eforms[3]$ has to fulfill}
  R_{\tilde{a}bc} &:= (M^{-1})^d_{\tilde a}R_{dbc}, \\
  0 &= M^{\tilde e}_{a],[[k}R_{\tilde{e}bc],[[l} +2M^{\tilde e}_{a],[[k}(M^{-1})^f_{\tilde e}R_{\tilde{g}f[c,l]]}M^{\tilde{g}}_{b]} +M^{\tilde e}_{[a,[[k}R_{\tilde{e}[bc,[[l}
\end{align*}  and $[\dots]$ and $[[\dots]]$ mean skew-symmetrization in the embraced indices.

\begin{lemma}  The above short exact sequences permit splittings (in the category of pointed fiber bundles).  A splitting $i\:\E\to\M$ induces a map $p_i\:\M\to T^*[3]M$, and a splitting $j\:E[1]\to\E$ induces a map $p_j\:\E\to E^*[2]$.
\begin{align}
   0\times M\leftarrow T^*[3]M\xleftarrow{p_i} \M\xleftarrow{i} \E\leftarrow 0\times M  \\
   0\times M\leftarrow E^*[2]\xleftarrow{p_j} \E\xleftarrow{j} E[1]\leftarrow 0\times M \nonumber
\end{align}

A change of the splitting $j$ is a section $R\in\Esforms[1]{E}$ which modifies $p_j^*$ to $\psi\mapsto p_j^*(\psi)+R(\psi)$.  A change of the splitting $i$ for fixed splitting $j$ is a section $L\in\Omega^1_M(TM,\wedge^3E^*\oplus E^*\otimes E)$ which modifies $p_i^*$ to $X\mapsto p_i^*(X)+L(X)$ for $X\in\Gamma(TM)$.
\end{lemma}
\begin{proof} The fibers are contractible.
\end{proof}


\begin{ex}\label{ex:1} The Q-structure $\Q$ of a PQ3 manifold has the following components:
\begin{gather}\begin{split}
  \Q &= \rho^i_a(x) \theta_i\xi^a +\tfrac12C_{ab}^c(x)\xi^a\xi^b b_c +\tfrac1{4!}h_{abcd}(x)\xi^a\xi^b\xi^c\xi^d +\tfrac12B^{ab}(x) b_a b_b
\end{split}
\intertext{i.e.\ a map}
 \rho\: E \to TM : \rho(\psi)[f] = \{\,\{\Q,p_j^*\psi\},f\} \;,
\intertext{a symmetric 2-vector}
 B\in\Gamma(S^2E) : B(\alpha,\beta)=\{\{\Q,q^*\alpha\},q^*\beta\} \;,
\intertext{an E4-form}
 h \in \Eforms[4] : h(\psi_1,\psi_2,\psi_3,\psi_4)=\{\{\{\{\Q, p_j^*\psi_1\}, p_j^*\psi_2\}, p_j^*\psi_3\}, p_j^*\psi_4\} \;,
\intertext{and a bracket}\label{dskewBr}
 [.,.]\:\Gamma(E)\otimes\Gamma(E)\to \Gamma(E) : [\phi,\psi] = i_q^*\{\{\Q,p_j^*\phi\},p_j^*\psi\}\;.
\end{gather}
The nilpotence $\{\Q,\Q\}$ is equivalent to
\begin{align}  2\rho^i_{[a}\rho^j_{b],i} &= \tfrac12\rho^j_c C_{ab}^c & [\rho(\phi),\rho(\psi)]&=\rho[\phi,\psi]  \\
  \rho^i_{[a}C^d_{bc],i}+C_{ab]}^eC_{e[c}^d +\tfrac1{3!}&h_{abce}B^{ed} = 0 \label{Jac}\\
    [\phi,[\psi_1,\psi_2]] -[[\phi,\psi_1],\psi_2]& -[\psi_1,[\phi,\psi_2]]& = B^\#\circ h(\phi,&\,\psi_1,\psi_2)  \nonumber\\
  \tfrac1{4!}\rho^i_{[a}h_{bcde],i} +\tfrac1{12}C_{ab]}^f h_{f[cde} &= 0 &  \uD h &=0 \\
  \rho^i_a B^{bc}_{,i}+C_{ad}^{b)}B^{d(c} &= 0 &  \uD B &= 0  \\
  \rho^i_a B^{ab} &= 0 &  \rho\circ B^\# &= 0
\end{align}  where $\phi,\psi,\psi_i\in\Gamma(E)$
\end{ex}
\begin{proof}  In order to show, e.g., \eqref{Jac} use adapted coordinates such that $p_j^*(\theta_a)=\theta_a$ and note that the derived bracket encodes:
\begin{align*}
 [\phi,\psi] &:=i_q^*\{\{\Q,p_j^*\phi\}, p_j^*\psi\} = \{\{\Q,p_j^*\phi\},p_j^*\psi\}-p^*q^*h(\phi,\psi)
\end{align*}  Then use computation in these coordinates.
\end{proof}

\begin{ex}[Cotangent structure of a Courant algebroid]\label{ex:T^*EE}  Due to Roytenberg \cite{Royt02} a Courant algebroid structure on the vector bundle $A\to M$ is in one-to-one correspondence to a cubic Hamiltonian $\Theta_A$ on the symplectic realization $\A:=A[1]\times_{(A\oplus A^*)[1]}T^*[2]A[1]$.  The Dorfman bracket can be expressed as a derived bracket.  Analog to Roytenberg's example $T^*[2]T[1]M$ one can also consider $\M:=T^*[3]\A$ with a Hamiltonian lift of the $Q$-structure.  The role of $E$ in Proposition~\ref{p:PQ3} is played by a vector bundle $E$ that fits into the short exact sequence of vector bundles over $M$
 $$ 0\times M\to TM \to E \xrightarrow{\pi} A \to 0\times M $$ where $\pi$ is the projection $T^*[3]\A\to\A$.
If one follows the steps of Example~\ref{ex:1}, after choosing a splitting of the short exact sequences, one recovers an H-twisted Lie algebroid on $E$.  Note that the lifted Q-structure mixes the $A$ and the $TM$-component in $E$, however $\Gamma(A)$ is a subalgebra with an almost Lie algebroid bracket deriving from the Courant bracket and a compatible connection.  In coordinates -- $\xi^i$ on $T[1]M$, $b_i$ on $T^*[2]M$, and $\theta_i$ on $T^*[3]M$ -- the Q-structures read as follows:
\begin{align*}
  \Theta_A =& \rho^i_a(x)\xi^a b_i +\tfrac16C_{abc}(x)\xi^a\xi^b\xi^c \\
  Q_A =& \rho^i_a\xi^a\pfrac{}{x^i}+\rho^i_ag^{ab}b_i\pfrac{}{\xi^b}+\rho^i_{a,j}\xi^a b_i\pfrac{}{b_j} +\tfrac12C_{abc}g^{cd}\xi^a\xi^b\pfrac{}{\xi^d} \\&+\tfrac16C_{abc,i}\xi^a\xi^b\xi^c\pfrac{}{b_i}
\intertext{and the Hamiltonian lift}
  \Theta =& \rho^i_a\xi^a\theta_i +\rho^i_ag^{ab}b_i b_b +\rho^i_{a,j}\xi^a b_i \xi^j +\tfrac12C_{abc}g^{cd}\xi^a\xi^b b_d +\tfrac16 C_{abc,i}\xi^a\xi^b\xi^c \xi^i
\intertext{Therefore the symmetric bivector $B$ and the $E$-4-form $h$ compute as}
  B\:&E^*_x\cdot E^*_x\to \RR : (\alpha,\beta) \mapsto \rho^i_ag^{ab}(\alpha_i\beta_b+\alpha_b\beta_i), \\
  h(&\phi,\psi,\chi,\nu) = C_{abc,i}(\phi^a\psi^b\chi^c\nu^i+\alt),
\intertext{The anchor map factors through $\pi$}
  \rho_E\:&E\to TM : \psi \mapsto \rho_A(\pi(\psi)) \;,
\intertext{and the skew-symmetric bracket in coordinates is}
  [\phi,\psi] &= i_q^*\{\{\Theta,\phi\},\psi\} = \rho(\phi)[\psi^a]b_a +C_{ab}^c\phi^a\psi^b b_c -\rho(\psi)[\phi^a]b_a \\&\quad\quad+\rho(\phi)[\psi^i]b_i +\rho^i_{a,j}(\phi^a\psi^j -\phi^j\psi^a)b_i -\rho(\psi)[\phi^i]b_i \;.
\end{align*}   Note that this bracket fulfills Leibniz rule and has Jacobiator $B^\#\circ\tilde{h}$.
\end{ex}

\subsection{The splittable case}
\begin{defn}  We call an H-twisted Lie algebroid $(E,\rho,[.,.],H)$ splittable if there exists a $\uD$-closed E4-form $h\in\Eforms[4]$ and a $\uD$-closed symmetric 2-vector $B\in\Esforms[0]{S^2\ker\rho}=\Gamma(S^2\ker\rho)$ such that
 $$ H = B^\#\circ \tilde{h} \;. $$
\end{defn}
The above Example~\ref{ex:1} is obviously splittable, however Example~\ref{ex:rk3} cannot be split for degree reasons.  In general we arrive at the following theorem.

\begin{thm}\label{thm:PQ3} An H-twisted Lie algebroid can be written as in Example~\ref{ex:1} iff it is splittable.

There is a one-to-one correspondence between split H-twisted Lie algebroids and PQ3-manifolds with splitting.
\end{thm}

\begin{defn}  We define the \emph{split cohomology} of a splittable H-twisted Lie algebroid as the cohomology of $\gsmooth(\M)$ under the differential $Q:=\{\Q,.\}$.
\end{defn}
\begin{rem}  In analogy to the Courant algebroid the lowest orders of cohomology have the following interpretations: $H^0(E)$ are the smooth functions on $M$ that are constant along the integral leaves of the image of $\rho$.  $q^*H^1(E)$ are the $\uD$-closed E-1-forms modulo $\uD$-exact 1-forms. $H^2(E)$ are the infinitesimal automorphisms of the split $E$ modulo the inner automorphisms $\ELie_\psi:=[\psi,.]$ (for $\psi\in\ker\Ht\cap\Gamma(E)$).  $H^3(E)$ are the obstructions of extending an infinitesimal automorphism to a formal one.
\end{rem}

\begin{cor}\label{c:tc}  Given a split $H$-twisted Lie algebroid there is a complex
 \[  0\to T^*M \xrightarrow{\rho^T} E^* \xrightarrow{B^\#} E \xrightarrow{\rho} TM \to 0 \;.
 \]
\end{cor}
\begin{proof}  This is the so-called tangent complex of a Q-manifold. The vector spaces are the lowest degree tangent spaces in the sequence
 $$ \M \to \E\to E[1]\to M $$ and the map is the commutator with $Q$ which turns out to be $\smooth(M)$-linear.  The abstract reason why the sequence is a complex is that $[Q,Q]=0$, but in our particular case we also see that $B\in\Gamma(S^2\ker\rho)$ and therefore $\rho\circ B^\#=0$, as well as the dual sequence $0=B^{\#T}\circ \rho^T=B^\#\circ\rho^T$.
\end{proof}

\section{Regular cohomology}
The $H$-twisted Lie algebroid where the anchor has constant rank is also a Lie-2 algebroid in the sense of \cite{GSr10}.  The identifications are according to the names of the maps and $t\:\ker\rho\hookrightarrow E$ the embedding.  Therefore there is a realization of the structure functions as $Q$-structure on a graded manifold of degree 2.

\begin{prop}  Let $(E,\rho,[.,.],H)$ be an $H$-twisted Lie algebroid with anchor map $\rho$ of constant rank.  Then there is a nilpotent vector field $Q$ of degree 1 on the graded manifold $\M:=(\ker\rho)[2]\oplus E[1]$ with $l\:\Gamma(E)\xrightarrow{\sim}\X_{[-1]}(E[1])\subset\X_{[-1]}(\M)$ and $l'\:\Gamma(\ker\rho)\xrightarrow{\sim}\X_{[-2]}(\ker\rho[2])\cong\X_{[-2]}(\M)$ such that
\begin{align*}
  l[\psi_1,\psi_2] &= \pr_2[[l\psi_1,Q],l\psi_2],  \\
  \rho(\psi)[f] &= [l\psi,Q][f]\in\gsmooth_{[0]}(\M)=\smooth(M), \\
  l' H(\psi_1,\psi_2,\psi_3) &= [[[Q,l\psi_1],l\psi_2],l\psi_3] \in\X_{[-2]}(\M)=\Gamma(\ker\rho),  \\
  l' \Econn_\psi \phi &= [[Q,l\psi],l'\phi] \in\X_{[-2]}(\M)\quad\text{for }\phi\in\Gamma(\ker\rho)
\end{align*}
\end{prop}
\begin{proof}  This is a corollary of Proposition~3.1 in \cite{GSr10}.  The coordinate description of the Q-structure is
\begin{align*} Q =& \rho_i^a(x)\xi^a\pfrac{}{x^i} -\tfrac12C_{ab}^c(x)\xi^a\xi^b\pfrac{}{\xi^c} +t^a_B(x)b^B\pfrac{}{\xi_a} -\Gamma_{aB}^C\xi^a b^B\pfrac{}{b^C} \\&+\tfrac16H_{abc}^B\xi^a\xi^b\xi^c\pfrac{}{b^B}
\end{align*}  where $\Gamma_{aB}^C$ are the connection coefficients of $\Enabla$, $C_{ab}^c$ the structure functions of the bracket, $x^i$ coordinates on $M$, $\xi^a$ fiber-coordinates on $E[1]$, and $b^B$ fiber-coordinates on $(\ker\rho)[2]$.
 The nilpotence of $Q$ is equivalent to the axioms of an $H$-twisted Lie algebroid (for these choices of $\M$ and $\Enabla$ induced by $[.,.]$).
\end{proof}

Therefore we can define:
\begin{defn} The regular cohomology of an $H$-twisted Lie algebroid is the cohomology of the vetor fields $\X_{\bullet}(\M)$ under the differential $[Q,.]$.
\end{defn}


\end{document}